\documentclass[12pt]{article}
\usepackage{amsthm}
\usepackage{amsmath,amssymb}
\usepackage{mathtools,hyperref}

\theoremstyle{plain}
\numberwithin{equation}{section}

\newtheorem{theorem}{Theorem}[section]
\newtheorem{corollary}[theorem]{Corollary}
\newtheorem{lemma}[theorem]{Lemma}

\newtheorem{remark}{Remark}[section]

\newtheorem{conjecture}{Conjecture}
\DeclareMathOperator{\diag}{diag}
\DeclareMathOperator{\spec}{spec}
\DeclareMathOperator{\sgn}{sgn}
\newcommand{\editor}{Renata Del-Vecchio}
\newcommand{\ArticleType}{Research Articles} 
\newcommand{\RecievedDate}{Oct 1, 2021}
\newcommand{\RevisedDate}{Jan 7, 2022}
\newcommand{\AcceptedDate}{Jan 7, 2022}
\newcommand{\PublishedDate}{TBA}
\newcommand{\JournalIndex}{Volume 1 (2022), Pages 20--39}
\newcommand{\LastName}{Kannan}
\newcommand{\JournalIndexShort}{\LastName\, et al./ American Journal of Combinatorics 1 (2022) 20--39}
\usepackage{fancyhdr}
\usepackage{amsmath,amsthm,amssymb,mathtools,url,graphicx,pdfpages,tikz,rotating}
\usepackage[affil-it]{authblk}
\usepackage[left=1in,right=1in,top=1.2in, bottom=1in]{geometry}

\theoremstyle{definition}

\numberwithin{equation}{section}

\begin{document}
\setcounter{page}{1}
\noindent {\color{teal}\bf\large American Journal of Combinatorics} \hfill \ArticleType\\
\JournalIndex

\title{Normalized Laplacians for Gain Graphs}

\author{M. Rajesh Kannan}
\author{Navish Kumar}
\author{Shivaramakrishna Pragada} 
\affil{\normalsize\rm (Communicated by \editor) \vspace*{-24pt}}
\date{}

{\let\newpage\relax\maketitle}

\begin{abstract}
 We propose the notion of normalized Laplacian matrix $\mathcal{L}(\Phi)$ for a gain graph $\Phi$ and study its properties in detail, providing insights and counterexamples along the way. We establish bounds for the eigenvalues of $\mathcal{L}(\Phi)$ and characterize the classes of graphs for which equality holds. The relationships between the balancedness, bipartiteness, and their connection to the spectrum of $\mathcal{L}(\Phi)$ are also studied. Besides, we extend the edge version of eigenvalue interlacing for the gain graphs.  Thereupon, we determine the coefficients for the characteristic polynomial of  $\mathcal{L}(\Phi)$.
\end{abstract}

\renewcommand{\thefootnote}{\fnsymbol{footnote}} 
\footnotetext{\hspace*{-22pt} MSC2020: 05C22, 05C50, 94A12;
Received \RecievedDate; Revised \RevisedDate; Accepted \AcceptedDate; Published \PublishedDate \\
\copyright\, The author(s). Released under the CC BY 4.0 International License}
\renewcommand{\thefootnote}{\arabic{footnote}}

\section{\textsf{Introduction}}Spectral graph theory is the study of the properties of a graph related to the characteristic polynomial, eigenvalues, and eigenvectors of matrices associated with the graph, such as the adjacency matrix,  Laplacian matrix, and so on \cite{bapat2010graphs,brouwer2011spectra,chung1996lectures,intro_graph_spectra_2009,cvetkovic2007signless}. Several structural properties of graphs are deduced from the eigenvalues of these matrices.  For example, the number of edges (via the adjacency, Laplacian, and signless Laplacian), the number of connected components (via the Laplacian and normalized Laplacian), bipartiteness (via the adjacency and normalized Laplacian), and the number of bipartite components (via the signless Laplacian and normalized Laplacian), etc.

The normalized Laplacian matrices for both undirected and directed graphs are well-studied matrix classes in spectral graph theory. The eigenvalues and eigenvectors of the normalized Laplacian matrices reveal several combinatorial properties of the underlying graphs. In particular, the second smallest eigenvalue of the normalized Laplacian is useful in studying the mixing rate of random walks, expansion of a graph, Cheeger constant, etc.  For more details, we refer to the monograph by Chung \cite{chung1996lectures}. The normalized Laplacian for directed graphs are studied by  Bauer \cite{bauer2012normalized} and Chung \cite{dir-lap-chung}.

A \emph{directed graph (or digraph)} $ X $ is an ordered pair $ (V(X), E(X)) $, where\break $ V(X)=\{ v_{1}, v_{2}, \dots,v_{n}\} $ is the vertex set and $ E(X) $ is the directed edge set. A directed edge from the vertex $ v_{s} $ to the vertex $ v_{t} $ is denoted by $ \overrightarrow{e_{st}} $. If $ \overrightarrow{e_{st}} \in E(X)$ and $  \overrightarrow{e_{ts}}\in E(X)$, then the pair $ \{v_{s},v_{t}\} $ is called a \emph{digon} of $ X $. The  underlying graph of $ X $ is a simple undirected graph obtained from $ X $ by replacing a directed edge by an undirected edge and it is denoted by $ \Gamma(X) $.  The  \textit{Hermitian adjacency matrix}  of a digraph $ X $ is  denoted by $H(X)$ and is defined as follows:
$$\mbox{$ (s,t)$-{th} entry of }H(X)=h_{st}=\begin{cases}
	1& \text{if  } \mbox{$\overrightarrow{e_{st}}\in E(X) $ \text{and} $\overrightarrow{e_{ts}} \in E(X)$},\\
	i& \text{if  } \mbox{$\overrightarrow{e_{st}} \in E(X)$ \text{and} $\overrightarrow{e_{ts}} \notin E(X)$},\\
	-i& \text{if  } \mbox{$\overrightarrow{e_{st}} \notin E(X)$ \text{and} $\overrightarrow{e_{ts}} \in E(X)$},\\
	0&\text{otherwise.}\end{cases}$$ This was introduced by Guo and Mohar \cite{GuoMohar} and Liu and Li
\cite{liu2015hermitian}.  In a very recent paper \cite{hermitian}, Yu et al.  studied  the notion of Hermitian normalized Laplacian matrix.

For a given group $\mathfrak{G}$, a $\mathfrak{G}$-gain graph is a graph $G$ with each orientation of an edge of $G$ is assigned an element $g \in \mathfrak{G}$ (called a gain of the oriented edge)  and whose inverse $g^{-1}$ is assigned to the opposite orientation of the edge. The notion of the $\mathfrak{G}$-gain graph was introduced by Zaslavsky \cite{Zas1982Signed, zaslavsky1989biased}. Let $\mathbb{T}=\{z\in\mathbb{C}:|z|=1\} $ be the multiplicative group of unit complex numbers. If $\mathfrak{G} = \mathbb{T}$,  we call $G$  as a $\mathbb{T}$-gain graph(or a gain graph). Note that the Hermitian adjacency matrix can be considered as the adjacency matrix of a $\mathbb{T}$-gain graph where the gains are from $ \{1, \pm i\}$.
In 2012, Reff introduced the notion of the adjacency matrix and Laplacian matrix of a gain graph canonically \cite{reff2012spectral}.  Afterward, Mehatari et al. studied several spectral properties of  gain adjacency matrices \cite{ran-raj-ani-LAMA}.

In this article, we define the notion of gain normalized Laplacian matrix for a  gain graph. We aim to study some of the basic properties of gain normalized Laplacian matrix, and to establish the connections between its eigenvalues and the structural properties of the underlying graph. Many results from the papers mentioned above have been extended here in the context of gain normalized Laplacian matrix and a complete proof of many new results, along with counter examples for results that do not follow, have been provided. We start by defining the gain normalized Laplacian $\mathcal{L}(\Phi)$, analogous to the Hermitian Laplacian defined in \cite{hermitian}. Then, we study the properties of the spectrum of $\mathcal{L}(\Phi)$ and characterize its eigenvalues to establish a relation between structural properties of the underlying graph $G$. We then obtain bounds for eigenvalues of $\mathcal{L}(\Phi)$, and characterize, in terms of both structure of the graph and gains of the edges,  the classes of graphs for which the inequality is sharp. Thereupon, we study the relationship between the balancedness,  bipartiteness, and spectral radius of the normalized Laplacian associated with a graph. On top of that, we investigate the symmetry of the eigenvalues of $\mathcal{L}(\Phi)$, and provide an edge-version of the eigenvalue interlacing result. We finish our theoretical exposition by presenting two expressions for the coefficients of the characteristic polynomial of $\mathcal{L}(\Phi)$.

This article is organized as follows:  In Section \ref{prelim}, we include some of the  known results which are useful for this work. In Section \ref{sec:norm-gain-lap}, we start by defining the gain normalized Laplacian matrix for a gain graph and present some of its basic properties with a significant focus on spectral and balancedness related properties. In Section \ref{sec:bala-bipa-eig}, an equivalent condition for the equality of set of eigenvalues of $\mathcal{L}(\Phi)$ and their connections with the structure of the underlying graph is provided. Next, we provide an edge version of the eigenvalue interlacing result in section \ref{sec:eig-inter}. In Section \ref{sec:char-poly}, the coefficients of the characteristic polynomial of $\mathcal{L}(\Phi)$ are determined.

\section{\textsf{Preliminaries}}\label{prelim}
Let $G = (V,E)$ be a simple, undirected, finite graph with the vertex set $V(G) = \{v_1,\dots,v_n\}$ and the edge set $E(G) \subseteq V \times V$. If two vertices $v_i$ and $v_j$  are adjacent, we write $v_i \sim v_j$, and the edge between them is denoted by $e_{ij}$, i.e., $e_{ij} = (v_i,v_j) \in E(G)$. The degree of the vertex $v_j$ is denoted by $d_j$. The
$(0, 1)$-adjacency matrix or simply the adjacency matrix of $G$ is an $n \times n$ matrix, denoted by $\mathbf{A}(G) = (a_{ij}) \in \mathbb{R}^{n \times n}$, whose rows and columns are indexed by the vertex set of the graph and the entries are defined by

\begin{equation}
	a_{ij} = \begin{cases} 1 & \quad \text{if} \quad v_i \sim v_j \\ 0 & \quad \text{otherwise.} \end{cases}
\end{equation}

We define a diagonal matrix $\mathbf{D}(G) = \diag(d_1, d_2, \dots, d_n )$, where $d_i$ is the degree of vertex $v_i$ in the underlying graph $G$ and the normalized adjacency matrix is defined as $\mathcal{A}(G) = \mathbf{D}^{-\frac{1}{2}} \mathbf{A}(G) \mathbf{D}^{-\frac{1}{2}}$. The (combinatorial) Laplacian matrix of a graph $G$ is defined as $\mathbf{L}(G) = \mathbf{D}(G)-\mathbf{A}(G)$.  The normalized Laplacian of a graph $G$, without isolated vetrices,  is defined as $  \mathcal{L}(G)  = \mathbf{D}^{-\frac{1}{2}}(G)\: \mathbf{L}(G)  \mathbf{D}^{-\frac{1}{2}}(G) = \mathbf{I} - \mathbf{D}^{-\frac{1}{2}}(G) \mathbf{A}(G) \mathbf{D}^{-\frac{1}{2}}(G)$. It is clear that $\mathcal{L}(G)$ is symmetric positive semi-definite. For further theory and applications related to the graph Laplacians, we refer to \cite{cavers2010normalized,chung1996lectures}.

For any simple graph $ G $, each undirected edge $ e_{st} \in E(G) $ is associated with a  pair of oriented edges, namely $ \overrightarrow{e_{st}} $ and $ \overrightarrow{e_{ts}} $. The set of all such oriented edges of a simple graph $ G $ is known as \textit{the oriented edge set} of $ G $, and is denoted by $ \overrightarrow{E}(G) $. A \textit{ $\mathbb{T}$-gain graph (or a gain graph)} on a simple graph $ G $ is a triplet $\Phi=(G,\mathbb{T},\varphi)$ such that  the map (\emph{the gain function}) $\varphi:\overrightarrow{E}(G)\rightarrow\mathbb{T}$  satisfies $\varphi(\overrightarrow{e_{st}})=\varphi(\overrightarrow{e_{ts}})^{-1}$. That is, for an oriented edge $\overrightarrow{e_{st}}$, if we assign a value $g$ (the \textit{gain} of the edge $ \overrightarrow{e_{st}}$) from $\mathbb{T}$, then assign $g^{-1 }$ to the oriented edge $\overrightarrow{e_{ts}}$. For simplicity, we use $\Phi=(G,\varphi)$ to denote a $\mathbb{T}$-gain graph instead of $\Phi=(G,\mathbb{T},\varphi)$, and call $ \varphi $  a $ \mathbb{T}$-gain on $ G $ if $ \Phi=(G, \varphi) $ is a $ \mathbb{T} $-gain graph on $ G $. In \cite{reff2012spectral},  Reff studied the notion of the adjacency matrix $\textbf{A}(\Phi)=(a_{st})_{n\times n}$ of a $\mathbb{T}$-gain graph $\Phi$. The entries of $\textbf{A}(\Phi)$ are given by
$$a_{st}=\begin{cases}
	\varphi(\overrightarrow{e_{st}})&\text{if } \mbox{$v_s\sim v_t$},\\
	0&\text{otherwise.}\end{cases}$$
It is clear that the matrix $\textbf{A}(\Phi)$ is Hermitian, and hence its eigenvalues are real. When $\varphi(\overrightarrow{e_{st}})=1$ for all $\overrightarrow{e_{st}}$, then $\textbf{A}(\Phi)=\textbf{A}(G)$. Thus we can consider $G$ as a $\mathbb{T}$-gain graph and we write this by $(G,1)$. By slight abuse of notation, we sometimes write $\varphi(e_{ij})$ as only $\varphi_{ij}$. The Laplacian matrix $\mathbf{L}(\Phi)$ is defined as $\mathbf{L}(\Phi) = \mathbf{D}(G)-\mathbf{A}(\Phi)$, where $\mathbf{D}(G) = \diag(d_1, d_2, \dots, d_n )$ is a diagonal matrix and $d_i$ is the degree of vertex $v_i$ in the underlying graph $G$. It is known that $\mathbf{L}(\Phi)$ is Hermitian and positive semi-definite.

The gain of a cycle (with some orientation) $C = v_1 v_2 \ldots v_l v_1$, denoted by $\varphi(C)$, is defined as the product of the gains of its edges, that is $$\varphi(C) = \varphi(e_{12}) \varphi(e_{23}) \cdots \varphi(e_{(l-1)l}) \varphi(e_{l1}).$$
A cycle $C$ is said to be neutral if $\varphi(C) = 1$, and a gain graph is said to be balanced if all its cycles are neutral. For a cycle $C$ of $G$, we denote the real part of the gain of $C$ by $\Re(\varphi(C))$, and it is independent of the orientation.
A function from the vertex set of $G$ to the complex unit circle $\mathbb{T}$ is called a switching function. Two gain graphs $\Phi_1 = (G, \varphi_1)$ and $\Phi_2 = (G, \varphi_2)$ are switching equivalent, written as $\Phi_1 \sim \Phi_2$, if there is a switching function $\zeta: V \to \mathbb{T}$ such that $$\varphi_2(e_{ij})=\zeta(v_i)^{-1}\varphi_2(e_{ij})\zeta(v_j).$$

The switching equivalence of two gain graphs can be defined in the following equivalent way: Two gain graphs $\Phi_1 = (G, \varphi_1)$ and $\Phi_2 = (G, \varphi_2)$ are switching equivalent, if there exists
a diagonal matrix $\mathbf{D}_{\zeta}$ with diagonal entries from $\mathbb{T}$, such that
\begin{align}\label{eq: switching equi}
	\mathbf{A}(\Phi_2) = \mathbf{D}_{\zeta}^{-1}\mathbf{A}(\Phi_1)\mathbf{D}_{\zeta}.
\end{align}
Switching equivalence preserves connectivity and balancedness.
We write $\Phi\sim(G,1)$, if $\Phi$ is switching equivalent to the gain which assigns $1$ to all the edges of  $G$.
\begin{theorem}[\cite{zaslavsky1989biased}]\label{balan-equi-switch}
	Let $\Phi=(G,\varphi)$ be a $\mathbb{T}$-gain graph. Then          $\Phi$ is balanced if and only if $\Phi\sim(G,1)$,
\end{theorem}
\begin{theorem}[{\cite[Theorem 4.5]{ran-raj-ani-LAMA}}]\label{theorem:bipa-balan}
	Let $ G $ be a connected graph. Then
	\begin{enumerate}
		\item[(i)] If $ G $ is bipartite, then whenever $ \Phi $ is balanced implies $-\Phi $ is balanced.
		\item[(ii)] If $ \Phi $ is balanced implies $ -\Phi $ is balanced for some gain $ \Phi $, then the graph is bipartite.
	\end{enumerate}
\end{theorem}

For a matrix $\mathbf{A} = (a_{ij}) \in \mathbb{C}^{n \times n}$, define $\vert \textbf{A}\vert = (|a_{ij}|)$. Let $\rho(\mathbf{A})$ and $\spec(\mathbf{A})$ denote the spectral radius and the set of eigenvalues of $\mathbf{A}$, respectively. The following results about nonnegative matrices will be useful throughout the article.

\begin{theorem}[{\cite[Theorem 8.1.8]{horn2012matrix}}]\label{theorem: rho}
	Let $\mathbf{A}, \mathbf{B} \in \mathbb{C}^{n \times n}$ and suppose that $\mathbf{B}$ is nonnegative. If $|\mathbf{A}| \leq \mathbf{B}$, then $\rho(\mathbf{A})\leq \rho(|\mathbf{A}|)\leq \rho(\mathbf{B}).$
\end{theorem}

\begin{theorem}[{\cite[Theorem 8.4.5]{horn2012matrix}}]\label{theorem: mat-sim-gain}
	Let $\mathbf{A}, \mathbf{B} \in \mathbb{C}^{n \times n}$. Suppose that $\mathbf{A}$ is nonnegative and irreducible, and $\mathbf{A} \geq |\mathbf{B}|$. Let $\lambda = e^{i\varphi}\rho(\mathbf{B})$ be a given maximum-modulus eigenvalue of $\mathbf{B}$. If $\rho(\mathbf{A})=\rho(\mathbf{B})$, then there is a diagonal unitary matrix $\mathbf{D} \in \mathbb{C}^{n \times n}$ such that $\mathbf{B} = e^{i\varphi} \mathbf{D} \mathbf{A} \mathbf{D}^{-1}$.
\end{theorem}

Next is the well-known Courant-Fisher theorem, which provides a variational formulation for the eigenvalue problem of the Hermitian matrices.
\begin{theorem}[\cite{meyer2000matrix}] \label{theorem:CF}
	Let $\bf H$ be an $n \times n$ Hermitian matrix with eigenvalues $\lambda_1 \leq \lambda_2 \leq \ldots \leq \lambda_n$. For an integer $k \: (1 \leq k \leq n)$, we have
	
	$$\lambda_k = \max_{\mathbf{x}^{(1)},\mathbf{x}^{(2)}, \ldots, \mathbf{x}^{(k-1)} \in \mathbb{C}^n} \: \min_{\substack{\mathbf{x} \perp \mathbf{x}^{(1)},\mathbf{x}^{(2)}, \ldots, \mathbf{x}^{(k-1)}; \\ \mathbf{x} \neq 0; \\ \mathbf{x}  \in \mathbb{C}^n}} \frac{\mathbf{x}^*\mathbf{H}\mathbf{x}}{\mathbf{x}^*\mathbf{x}},$$
	
	$$\lambda_k = \min_{\mathbf{x}^{(k+1)},\mathbf{x}^{(k+2)}, \ldots,
		\mathbf{x}^{(n)} \in \mathbb{C}^n} \max_{\substack{\mathbf{x} \perp \mathbf{x}^{(k+1)},\mathbf{x}^{(k+2)}, \ldots, \mathbf{x}^{(n)}; \\ \mathbf{x} \neq 0; \\ \mathbf{x}  \in \mathbb{C}^n}} \frac{\mathbf{x}^*\mathbf{H}\mathbf{x}}{\mathbf{x}^*\mathbf{x}},$$where $\mathbf{x}^{(1)},\mathbf{x}^{(2)}, \ldots, \mathbf{x}^{(k-1)}$ are linearly independent.
\end{theorem}
\section{\textsf{Normalized gain Laplacian matrices}}\label{sec:norm-gain-lap}

In this section we introduce the notion of the gain normalized Laplacian matrix, and  study some of the spectral properties of the $\mathbb{T}$-gain graphs.
The \textit{gain normalized Laplacian matrix} $\mathcal{L}(\Phi) = (\mathcal{L}_{ij} ) \in \mathbb{C}^{n \times n} $ is defined entry-wise by

\begin{equation}
	\mathcal{L}_{ij} = \begin{cases} 1 & \quad \text{if} \quad v_i=v_j \: \text{and} \: d_i \neq 0, \\ \smallskip -\frac{1}{\sqrt{d_i d_j}} & \quad \text{if} \quad v_i \sim v_j \: \text{and} \: v_j \sim v_i, \\ \smallskip -\frac{\varphi(e_{ij})}{\sqrt{d_i d_j}} & \quad \text{if} \quad v_i \sim v_j \: \text{and} \: v_j \not\sim v_i, \\ \smallskip -\frac{\overline{\varphi}(e_{ij})}{\sqrt{d_i d_j}} & \quad \text{if} \quad v_i \not\sim v_j \: \text{and} \: v_j \sim v_i ,\\ 0 & \quad \text{otherwise.} \end{cases}
\end{equation}

Next, we shall prove a couple of basic properties about the gain normalized Laplacian matrices. The following characterization of bipartiteness with the normalized adjacency spectrum is useful throughout the article.

\begin{lemma}\label{spec-sym-bip-nor-adj}
	For a connected graph $G$, $\spec{\mathcal{A}(G)} = \spec{\mathcal{A}(-G)}$ if and only if $G$ is bipartite.
\end{lemma}
\begin{proof}
	If $G$ is bipartite, then it is easy to see that $\spec{\mathcal{A}(G)} = \spec{\mathcal{A}(-G)}$. Conversely, let $\spec{\mathcal{A}(G)} = \spec{\mathcal{A}(-G)}$. We know that $0$ is an eigenvalue of $\mathcal{L}(G)$, and hence $1$ is an eigenvalue of $\mathcal{A}(G)$. By the assumption, $-1$ is an eigenvalue of $\mathcal{A}(G)$, and hence $2$ is an eigenvalue of  $\mathcal{L}(G)$. Thus, by \cite[Lemma 1.7]{chung1996lectures}, $G$ is bipartite.
\end{proof}

The following lemmas characterize switching equivalence in terms of the spectrum and spectral radius of associated gain matrices.

\begin{lemma}\label{lemma:spec-equ-adj-lap}
	Let $\Phi_1$ and $\Phi_2$ be two connected gain graphs. If $\Phi_1 \sim \Phi_2$, then the following statements hold:
	\begin{enumerate}
		\item $\spec(\mathbf{A}(\Phi_1)) = \spec(\mathbf{A}(\Phi_2))$.
		\item $\spec(\mathcal{A}(\Phi_1)) = \spec(\mathcal{A}(\Phi_2))$.
		\item $\spec(\mathbf{L}(\Phi_1)) = \spec(\mathbf{L}(\Phi_2))$.
		\item $\spec(\mathcal{L}(\Phi_1)) = \spec(\mathcal{L}(\Phi_2))$.
	\end{enumerate}
\end{lemma}

\begin{proof}
	Let $\mathbf{A}(\Phi_2) = \mathbf{D}_{\zeta}^{-1}\mathbf{A}(\Phi_1)\mathbf{D}_{\zeta}.$
	Then, we have $\spec(\mathbf{A}(\Phi_1)) = \spec(\mathbf{A}(\Phi_2))$. It is clear that, $\mathcal{A}(\Phi_2) = \mathbf{D}_\zeta^{-1}\mathcal{A}(\Phi_1)\mathbf{D}_\zeta$, and hence  $\spec(\mathcal{A}(\Phi_1)) = \spec(\mathcal{A}(\Phi_2))$. Now, as $\mathbf{D} -\mathbf{A}(\Phi_2) = \mathbf{D}_{\zeta}^{-1}(\mathbf{D}-\mathbf{A}(\Phi_1))\mathbf{D}_{\zeta}$, so $\spec(\mathbf{L}(\Phi_1)) = \spec(\mathbf{L}(\Phi_2))$. Also, \\ $   \mathcal{L}(\Phi_2) = (\mathbf{D}^{-\frac{1}{2}}\mathbf{D}_{\zeta}^{-1} \mathbf{D}^{\frac{1}{2}}) \: \mathcal{L}(\Phi_1) \: (\mathbf{D}^{\frac{1}{2}} \mathbf{D}_{\zeta} \mathbf{D}^{-\frac{1}{2}})$, hence $\mathcal{L}(\Phi_2) = \mathbf{D}_{\zeta}^{-1}\mathcal{L}(\Phi_1)\mathbf{D}_{\zeta}$ and  $\spec(\mathcal{L}(\Phi_1)) = \spec(\mathcal{L}(\Phi_2))$.
\end{proof}

\begin{lemma}\label{lemma: rho-adj-lap}
	Let $\Phi = (G, \varphi)$ be a connected $\mathbb{T}$-gain graph.  Then the following statements hold:
	\begin{enumerate}
		\item[(i)] $\rho(\mathbf{A}(\Phi)) \leq \rho(|\mathbf{A}(\Phi)|) = \rho(\mathbf{A}(G))$.
		\item[(ii)] $\rho(\mathcal{A}(\Phi)) \leq \rho(|\mathcal{A}(\Phi)|) = \rho(\mathcal{A}(G))$.
		\item[(iii)] $\rho(\mathcal{L}(\Phi)) \leq \rho(|\mathcal{L}(\Phi)|) = \rho(\mathcal{L}(-G))$.
	\end{enumerate}
\end{lemma}
\begin{proof}
	$(i)$ and $(ii)$ follow from Theorem \ref{theorem: rho} because  $|\mathbf{A}(\Phi)| \leq \mathbf{A}(G)$, and therefore  $|\mathcal{A}(\Phi)| \leq \mathcal{A}(G)$. $(iii)$  follows trivially because $\rho(\mathcal{L}(-G))=2.$
\end{proof}
The next two lemmas give the quadratic form of normalized Laplacian in terms of graph properties, which helps obtain results for the corresponding matrix spectrum.
\begin{lemma}{\cite[Lemma 5.3]{reff2012spectral}}   \label{lemma:unorm hermitian}
	Let $\Phi$ be a gain graph on $n$ vertices and $\mathbf{x} = (x_1 , x_2 , \dots, x_n) \in \mathbb{C}^n$ be a row vector. Then
	\begin{equation}
		\mathbf{x}^*\mathbf{L}(\Phi)\mathbf{x} = \sum_{v_i \sim v_j} |x_i-a_{ij} x_j|^2.
	\end{equation}
\end{lemma}

\begin{lemma}\label{lemma:normalized-quadratic-form}
	Let $\Phi$ be a connected gain graph. Then for every vector $\mathbf{x} \in \mathbb{C}^n$, the following holds
	\begin{equation}
		\mathbf{x}^{*} \mathcal{L}(\Phi) \mathbf{x} = \sum_{v_i \sim v_j} \left|\frac{x_i}{\sqrt{d_i}}- a_{ij} \frac{x_j}{\sqrt{d_j}}\right|^2.
	\end{equation}
\end{lemma}
\begin{proof}
	Let $ \mathbf{y} = \mathbf{D}^{-\frac{1}{2}} \mathbf{x}$.  Then $\mathbf{x}^{*} \mathcal{L}(G) \mathbf{x} = (\mathbf{D}^{-\frac{1}{2}}\mathbf{x})^{*} \: \mathbf{L}(G) \: (\mathbf{D}^{-\frac{1}{2}}\mathbf{x}) = \mathbf{y}^{*}\mathbf{L}(G) \mathbf{y}$. Thus, by Lemma \ref{lemma:unorm hermitian}, we have $\mathbf{x}^{*} \mathcal{L}(G) \mathbf{x} = \sum_{v_i \sim v_j} |y_i-a_{ij} y_j|^2$. Writing in terms of $\mathbf{x}$ yields the result.
\end{proof}

Let $\Phi$ be a connected gain graph. For complex column vectors $\mathbf{x}$ and $\mathbf{x}^{(i)}$, we define the vectors $$\mathbf{y} = \mathbf{D}^{-\frac{1}{2}} \mathbf{x}, \quad \mathbf{y}^{(i)} = \mathbf{D}^{-\frac{1}{2}} \mathbf{x}^{(i)}.$$ Note that $\mathbf{y} \perp \mathbf{y}^{(1)}, \mathbf{y}^{(2)}, \ldots, \mathbf{y}^{(k-1)} \quad \text{if and only if}  \quad \mathbf{x} \perp \mathbf{x}^{(1)},\mathbf{x}^{(2)}, \ldots, \mathbf{x}^{(k-1)}$. By Lemma \ref{lemma:unorm hermitian} and Theorem \ref{theorem:CF}, we have

\begin{equation*}
	\lambda_k = \max_{\mathbf{x}^{(1)},\mathbf{x}^{(2)}, \ldots, \mathbf{x}^{(k-1)} \in \mathbb{C}^n} \: \min_{\substack{\mathbf{x} \perp \mathbf{x}^{(1)},\mathbf{x}^{(2)}, \ldots, \mathbf{x}^{(k-1)}; \\ \mathbf{x} \in \mathbb{C}^n; ~\mathbf{x} \neq 0} } \frac{\mathbf{x}^*\mathcal{L}(\Phi)\mathbf{x}}{\mathbf{x}^*\mathbf{x}},
\end{equation*}

\begin{equation*}
	= \max_{\mathbf{x}^{(1)},\mathbf{x}^{(2)}, \ldots, \mathbf{x}^{(k-1)} \in \mathbb{C}^n} \: \min_{\substack{\mathbf{x} \perp \mathbf{x}^{(1)},\mathbf{x}^{(2)}, \ldots, \mathbf{x}^{(k-1)}; \\ \mathbf{x} \in \mathbb{C}^n;~ \mathbf{x} \neq 0} } \frac{(\mathbf{D}^{\frac{1}{2}} \mathbf{y})^*\mathbf{D}^{-\frac{1}{2}}\mathbf{L}(\Phi)\mathbf{D}^{-\frac{1}{2}}(\mathbf{D}^{\frac{1}{2}} \mathbf{y})}{(\mathbf{D}^{\frac{1}{2}} \mathbf{y})^*(\mathbf{D}^{\frac{1}{2}} \mathbf{y})},
\end{equation*}

\begin{equation*}
	= \max_{\mathbf{y}^{(1)},\mathbf{y}^{(2)}, \ldots, \mathbf{y}^{(k-1)} \in \mathbb{C}^n} \: \min_{\substack{\mathbf{y} \perp \mathbf{y}^{(1)},\mathbf{y}^{(2)}, \ldots, \mathbf{y}^{(k-1)}; \\ \mathbf{y} \in \mathbb{C}^n; ~\mathbf{y} \neq 0} } \frac{\mathbf{y}^*\mathbf{L}(\Phi)\mathbf{y}}{\mathbf{y}^*\mathbf{D}\mathbf{y}},
\end{equation*}

\begin{equation*}\label{eq:unnorm-rayleigh}
	= \max_{\mathbf{y}^{(1)},\mathbf{y}^{(2)}, \ldots, \mathbf{y}^{(k-1)} \in \mathbb{C}^n} \: \min_{\substack{\mathbf{y} \perp \mathbf{y}^{(1)},\mathbf{y}^{(2)}, \ldots, \mathbf{y}^{(k-1)}; \\ \mathbf{y} \in \mathbb{C}^n; ~\mathbf{y} \neq 0} } \frac{\sum_{v_i \sim v_j} |y_i-a_{ij} y_j|^2}{\sum_{i} d_i|y_i|^2}.
\end{equation*}
Similarly, we have

\begin{equation*}
	\lambda_k = \min_{\mathbf{y}^{(k+1)},\mathbf{y}^{(k+2)}, \ldots,
		\mathbf{y}^{(n)} \in \mathbb{C}^n} \max_{\substack{\mathbf{y} \perp \mathbf{y}^{(k+1)},\mathbf{y}^{(k+2)}, \ldots, \mathbf{y}^{(n)}; \\ \mathbf{y} \in \mathbb{C}^n; ~\mathbf{y} \neq 0} }\frac{\sum_{v_i \sim v_j} |y_i-a_{ij} y_j|^2}{\sum_{i} d_i|y_i|^2}.
\end{equation*}

In particular,
\begin{subequations}
	\begin{align}
		\lambda_1 &= \mathop{\min}_{\mathbf{y} \neq 0} \frac{\sum_{v_i \sim v_j} |y_i-a_{ij} y_j|^2}{\sum_{i} d_i|y_i|^2}, \mbox{~and} \label{eq:rayleighQ-a} \\
		\lambda_n &= \mathop{\max}_{\mathbf{y} \neq 0} \frac{\sum_{v_i \sim v_j} |y_i-a_{ij} y_j|^2}{\sum_{i} d_i|y_i|^2}. \label{eq:rayleighQ-b}
	\end{align}
\end{subequations}
The following two theorems establish bounds on the spectrum of gain normalized Laplacian matrix and the sharpness of the bounds are discussed in the next section.
\begin{theorem}\label{theorem:0-eig-2}
	Let $\Phi$ be a connected gain graph of order $n \geq 2$ .  If  $\{\lambda_1, \lambda_2, \ldots, \lambda_n\}$ is the spectrum of $\mathcal{L}(\Phi)$, then
	$$0 \leq \lambda_i \leq 2$$ for any $1 \leq i \leq n$.
\end{theorem}
\begin{proof}
	For two complex numbers $a, b$:  $|a + b|^2 \leq 2|a|^2 + 2|b|^2$ holds. Thus $$\sum_{v_i \sim v_j} |y_i-a_{ij} y_j|^2 \leq \sum_{v_i \sim v_j} (2|y_i|^2 + 2|y_j|^2) = \sum_{v_i \sim v_j} 2d_i|y_i|^2,$$ Then, it follows that $$0 \leq \frac{\sum_{v_i \sim v_j} |y_i-a_{ij} y_j|^2}{\sum_{i} d_i|y_i|^2} \leq 2.$$
\end{proof}

\begin{theorem}\label{theorem:connectedness-laplacian}
	Let $\Phi = (G, \varphi)$ be a connected $\mathbb{T}$-gain graph on $n$ vertices and let $0 = \lambda_1 \leq \lambda_2 \leq \ldots \leq \lambda_n$ be    the eigenvalues of $\mathcal{L}(\Phi)$. Then, the following statements hold:
	\begin{enumerate}
		\item[(i)] $\sum_{i=1}^n \lambda_i= n$.
		\item[(ii)] If $\Phi$ is a gain graph  and $n \geq 2$, then $ \lambda_2 \leq \frac{n}{n-1}.$ If $\Phi$ is balanced, then $ \lambda_n \geq \frac{n}{n-1}.$
		Moreover, equality in both cases holds if and only if $\Phi$ is a balanced complete graph.
		\item[(iii)] $\lambda_1 < 1$ and $\lambda_n > 1.$
	\end{enumerate}
\end{theorem}
\begin{proof}
	\begin{enumerate}
		\item[(i)] It follows from considering the trace of $\mathcal{L}(\Phi)$.
		\item[(ii)] Since all the eigenvalues are non-negative, it implies, $\sum_{i=2}^n \lambda_i \leq \sum_{i=1}^n \lambda_i = n$, which then implies $(n-1)\lambda_2 \leq n$.
		
		Let $\Phi$ be balanced. Then $\lambda_1=0$, implying $\sum_{i=2}^n \lambda_n \geq \sum_{i=1}^n \lambda_i = n$, which in turn implies $(n-1)\lambda_n \geq n$. Now if $G$ is complete, then by \cite[Lemma 1.7]{chung1996lectures}, $ \lambda_2 = \frac{n}{n-1}$. Conversely, suppose $\lambda_2 = \frac{n}{n-1}$. Then $\lambda_2 = \lambda_3 =\dots =\lambda_n = \frac{n}{n-1}$, and $\lambda_1 =0$. Thus, by Corollary \ref{corollary: singularity and balancedness}, $\Phi$ is balanced. In this case, the underlying graph $G$ is complete.
		
		\item[(iii)] Note that $n\lambda_1 \leq \sum_{i=1}^n \lambda_i = n$ and $n\lambda_n \geq \sum_{i=1}^n \lambda_i = n$. So $\lambda_1 \leq 1$ and $\lambda_n \geq 1.$ Assume that $\lambda_1=1$. Then $\lambda_1=\lambda_2= \cdots = \lambda_n=1$ since their sum equals to $n$. This implies that $\mathcal{L}(\Phi) = \mathbf{I}_n$ and $\mathbf{D}^{-\frac{1}{2}}(\Phi)\mathbf{A}(\Phi)\mathbf{D}^{-\frac{1}{2}}(\Phi) =0$. This is a contradiction to the fact that $\Phi$ is connected. Hence $\lambda_1<1$. Similarly, we have $\lambda_n>1$.
	\end{enumerate}
\end{proof}

\section{\textsf{Balancedness, bipartiteness and the eigenvalues of normalized gain Laplacian matrices of graphs }}\label{sec:bala-bipa-eig}
In this section we will establish the relationship between balancedness and bipartiteness of a gain graph and their connections with the spectra of gain normalized Laplacian.

\begin{theorem}\label{theorem: spec and (G,1)}
	Let $\Phi = (G, \varphi)$ be a connected $\mathbb{T}$-gain graph.  Then, $\spec(\mathcal{L}(\Phi)) = \spec(\mathcal{L}(G))$ if and only if $\Phi \sim (G,1)$.
\end{theorem}
\begin{proof}
	If $\spec(\mathcal{L}(\Phi)) = \spec(\mathcal{L}(G))$, then, it is easy to see that,  $\spec(\mathcal{A}(\Phi)) = \spec(\mathcal{A}(G))$. Thus, by Lemma \ref{lemma: rho-adj-lap} and Theorem  \ref{theorem: mat-sim-gain},  $\mathcal{A}(\Phi) = e^{i\theta} \mathbf{D}_\zeta \mathcal{A}(G)\mathbf{D}_\zeta^{-1},$ where $\mathbf{D}_\zeta$ is a unitary diagonal matrix, and hence $\mathbf{A}(\Phi) =  e^{i \theta} \mathbf{D}_\zeta \mathbf{A}(G) \mathbf{D}_\zeta^{-1}$. As both the matrices $\mathbf{A}(\Phi)$ and  $\mathbf{A}(G)$ are symmetric, so $\theta$ is either $0$ or $\pi$.  That is, either $\Phi$ is balanced or $-\Phi$ is balanced. If $\Phi$ is balanced, then, by Theorem \ref{balan-equi-switch}, we are done. If $-\Phi$ is balanced, then, by Lemma \ref{lemma:spec-equ-adj-lap}, $\spec(\mathcal{A}(G)) =  \spec(-\mathcal{A}(\Phi))$.  By the assumption, we have $\spec(\mathcal{A}(G)) =  \spec(-\mathcal{A}(G))$. Thus, by Lemma \ref{spec-sym-bip-nor-adj}, $G$ is bipartite. hence, by Theorem \ref{theorem:bipa-balan}, $\Phi$ is balanced.
	
	Conversely, if $\Phi \sim (G,1)$,  then, by Lemma \ref{lemma:spec-equ-adj-lap}, $\spec(\mathcal{L}(\Phi)) = \spec(\mathcal{L}(G))$.
\end{proof}
\begin{remark}
	Theorem \ref{theorem: spec and (G,1)} shows that converse of Lemma \ref{lemma:spec-equ-adj-lap} is true in the case of $\Phi_1 \sim (G,1)$, but this may not be true in general. That is, if $\spec(\mathcal{L}(\Phi_1)) = \spec(\mathcal{L}(\Phi_2))$, then the $\Phi_1$ and $ \Phi_2$ need not be switching equivalent. The counterexample is illustrated below.
	
	\rm Consider two complete gain graphs on $3$ vertices with the following adjacency matrices: $$\mathbf{A}(\Phi_1) = \begin{bmatrix} 0 & \text{i} & \frac{1+\text{i}}{\sqrt{2}} \\ -\text{i} & 0 & -\text{i} \\ \frac{1-\text{i}}{\sqrt{2}} & \text{i} & 0 \end{bmatrix} \quad \text{and}\quad \mathbf{A}(\Phi_2) = \begin{bmatrix} 0 & -\frac{1+\text{i}}{\sqrt{2}} & \text{i} \\ -\frac{1-\text{i}}{\sqrt{2}} & 0 & -\text{i} \\ -\text{i} & \text{i} & 0 \end{bmatrix}.$$
	Then, it is easy to check that $\spec(\mathbf{A}(\Phi_1)) = \spec(\mathbf{A}(\Phi_2))$, which implies that corresponding gain adjacency matrices are unitarily similar, i.e.,
	\begin{align}\label{eq:unitarily-similar}
		\mathbf{A}(\Phi_2) = \mathbf{U}\mathbf{A}(\Phi_1)\mathbf{U}^*
	\end{align}and they are related by the following unitary matrix relation:
	\begin{align*}
		\begin{bmatrix} 1 & 0 & 0 \\ 0 & 0 & -1 \\ 0 & 1 & 0 \end{bmatrix}
		\begin{bmatrix} 0 & \text{i} & \frac{1+\text{i}}{\sqrt{2}} \\ -\text{i} & 0 & -\text{i} \\ \frac{1-\text{i}}{\sqrt{2}} & \text{i} & 0 \end{bmatrix}
		\begin{bmatrix} 1 & 0 & 0 \\ 0 & 0 & 1 \\ 0 & -1 & 0 \end{bmatrix} = \begin{bmatrix} 0 & -\frac{1+\text{i}}{\sqrt{2}} & \text{i} \\ -\frac{1-\text{i}}{\sqrt{2}} & 0 & -\text{i} \\ -\text{i} & \text{i} & 0 \end{bmatrix}
	\end{align*}Now, suppose that there exists an unitary diagonal matrix $\mathbf{D} = \diag(d_1, d_2, d_3)$, such that Eq. \eqref{eq:unitarily-similar} holds. After putting $\mathbf{D}$ in Eq. \eqref{eq:unitarily-similar} and comparing coefficients, we get $$\begin{cases} d_1 \overline{d_2} = \frac{\text{i}-1}{\sqrt{2}} \\ d_1 \overline{d_3} = \frac{1+\text{i}}{\sqrt{2}} \\ d_2 \overline{d_3} =1 \end{cases},\quad \text{which, after eliminating}\: d_1,d_2\: \text{and}\: d_3, \:\: \text{gives}\: 1=-1,$$ a contradiction.
	Thus, if $\spec(\mathcal{L}(\Phi_1)) = \spec(\mathcal{L}(\Phi_2))$, then $\Phi_1$ and $\Phi_2$ need not be swithcing equivalent. Also, it is known that, if $\Phi_1 \sim \Phi_2$ then $\spec(\mathbf{A}(\Phi_1)) = \spec(\mathbf{A}(\Phi_2))$. The converse of this result was not considered in the literature. The above example shows that, the converse need not be true.
\end{remark}
It is known that, the Laplacian matrix of a gain graph is singular if and only if the gain is balanced \cite{wang2018gainlap}. Next result is a counterpart for the gain normalized Laplacian. The proof uses the Perron-Frobenius theorem, and also gives an alternate proof for the gain Laplacian case.
\begin{theorem}\label{theorem: nonsingular and balance}
	Let  $\Phi$ be a connected gain graph. Then,   $\mathcal{L}(\Phi)$ is nonsingular if and only if $\Phi$ is an unbalanced connected gain graph.
\end{theorem}
\begin{proof}
	If $\Phi$ is balanced, then $\mathbf{A}(G) = \mathbf{D}_\zeta^{-1}\mathbf{A}(\Phi)\mathbf{D}_\zeta$, where $\mathbf{D}_\zeta$ is an unitary diagonal matrix. Thus, $\mathbf{D}^{-\frac{1}{2}} \mathbf{A}(G) \mathbf{D}^{-\frac{1}{2}} = \mathbf{D}^{-\frac{1}{2}} \mathbf{D}_\zeta^{-1}\mathbf{A}(\Phi) \mathbf{D}_\zeta \mathbf{D}^{-\frac{1}{2}} = \mathbf{D}_\zeta^{-1} \mathbf{D}^{-\frac{1}{2}} \mathbf{A}(\Phi) \mathbf{D}^{-\frac{1}{2}}\mathbf{D}_\zeta$. Thus $\mathcal{A}(G)$ and $\mathcal{A}(\Phi)$ are similar, hence $\mathcal{L}(G)$ and $\mathcal{L}(\Phi)$ are similar. So, if $\Phi$ is balanced, then $\mathcal{L}(\Phi)$ is singular.
	
	We have $\mathcal{L}(\Phi) =  \mathbf{I} - \mathbf{D}^{-\frac{1}{2}} \mathbf{A}(\Phi) \mathbf{D}^{-\frac{1}{2}}.$ If $\lambda$ is an eigenvalue of $\mathcal{L}(\Phi)$, then $(\mathbf{I} - \mathbf{D}^{-\frac{1}{2}} \mathbf{A}(\Phi) \mathbf{D}^{-\frac{1}{2}}) \mathbf{x} = \lambda  \mathbf{x}$, for some non-zero vector $\mathbf{x}$. Thus, $\mathbf{D}^{-\frac{1}{2}} \mathbf{A}(\Phi) \mathbf{D}^{-\frac{1}{2}} \mathbf{x} = (1 - \lambda) \mathbf{x}$. So $\spec(\mathbf{D}^{-\frac{1}{2}} \mathbf{A}(\Phi) \mathbf{D}^{-\frac{1}{2}}) \subseteq [-1, 1]$. If $0$ is an eigenvalue of $\mathcal{L}(\Phi)$, then $1$ is an eigenvalue of $\mathcal{A}(\Phi)$. So, by Theorem \ref{theorem: mat-sim-gain}, $\mathcal{A}(\Phi) = \mathbf{D}\mathcal{A}(G)\mathbf{D}^{-1}$, and hence $\Phi$ is balanced.
\end{proof}
The following corollary is pivotal in proving several important results prevailing in the spectral theory for gain normalized Laplacian.
\begin{corollary}\label{corollary: singularity and balancedness}
	For a connected $\mathbb{T}$-gain graph $\Phi$, the following statements are equivalent.
	\begin{enumerate}
		\item[(i)] $\mathcal{L}(\Phi)$ is singular.
		\item[(ii)] $0$ is a simple eigenvalue of $\mathcal{L}(\Phi)$.
		\item[(iii)] $\Phi \sim (G,1)$.
	\end{enumerate}
\end{corollary}
\begin{proof}
	Proof follows  from Theorem  \ref{theorem: nonsingular and balance}. Note that, $0$ is a \textit{simple} eigenvalue of $\mathcal{L}(\Phi)$ because, by statement $(iii)$, we have $\spec(\mathcal{L}(\Phi)) = \spec(\mathcal{L}(G))$, then statement $(ii)$ follows from Theorem \ref{theorem:connectedness-laplacian} as $\Phi$ is connected.
\end{proof}

The next theorem provides a connection between the spectrum of $\mathcal{L}(\Phi)$ and $\mathcal{L}(-\Phi)$ leading to the corollaries that characterize the bipartiteness in terms of the balancedness.

\begin{theorem}\label{observation:eigenvalue of -Phi}
	Let $0 \leq \lambda_1 \leq \dots \leq \lambda_n$ be the eigenvalues of $\mathcal{L}(\Phi)$. If $0 \leq \alpha_1 \leq \alpha_2 \leq \ldots \leq \alpha_n \leq 2$ are the eigenvalues  of $\mathcal{L}(-\Phi)$, then
	
	\begin{enumerate}
		\item[(i)] $\alpha_i = 2-\lambda_{n-\text{i}+1}$.
		\item[(ii)] $\spec(\mathcal{L}(\Phi)) = \spec(\mathcal{L}(-\Phi))$ if and only if all eigenvalues of $\mathcal{L}(\Phi)$ (resp. $\mathcal{L}(-\Phi)$) are symmetric about $1$ (including multiplicities), i.e., for each eigenvalue $\lambda_i, 2-\lambda_i$ is also an eigenvalue of $\mathcal{L}(\Phi)$ (resp. $\mathcal{L}(-\Phi)$).
	\end{enumerate}
\end{theorem}

\begin{proof}
	\begin{enumerate}
		\item[(i)] Because $\mathcal{L}(\Phi)+\mathcal{L}(-\Phi)=2\mathbf{I}$.
		\item[(ii)] If $\spec(\mathcal{L}(\Phi)) = \spec(\mathcal{L}(-\Phi))$, then $\lambda_i=\alpha_i =2-\lambda_{n-\text{i}+1}$, and $\alpha_i = \lambda_i =  2-\alpha_{{n-\text{i}+1}}$.
		
		Conversely, suppose $\lambda_i=2-\lambda_{n-\text{i}+1}$ (resp. $\alpha_i=2-\alpha_{n-\text{i}+1}$). Since $\alpha_i=2-\lambda_{n-\text{i}+1}$ (resp. $\lambda_i=2-\alpha_{n-\text{i}+1}$), it implies that $\alpha_i=\lambda_i$, which completes the proof.
	\end{enumerate}
\end{proof}

\begin{remark}\label{observation:phi-balanced-2-eigenvalue}Let $\Phi = (G, \varphi)$ be a connected $\mathbb{T}$-gain graph. Then, it is easy to see that, $\rho(\mathcal{L}(-\Phi)) = \rho(\mathcal{L}(-G)) = 2$ if and only if $\Phi$ is balanced. Also,  $\rho(\mathcal{L}(\Phi)) = \rho(\mathcal{L}(-G))$ if and only if $-\Phi$ is balanced.
\end{remark}

\begin{conjecture}
	Let $\Phi = (G, \varphi)$ be a connected $\mathbb{T}$-gain graph. Then, $\rho(\mathcal{L}(\Phi)) = \rho(\mathcal{L}(G))$ implies  $\Phi$ is balanced. 
\end{conjecture}
We note that the converse true from Theorem \ref{theorem: spec and (G,1)}.

In the following theorem, we show that the above conjecture holds for bipartite graphs.
\begin{theorem}
	If $G$ is bipartite, then $\rho(\mathcal{L}(\Phi)) = \rho(\mathcal{L}(G))$ implies $\spec(\mathcal{L}(G)) = \spec(\mathcal{L}(\Phi))$ for every gain $\varphi$.
\end{theorem}

\begin{proof}
	Let $G$ be a bipartite graph, and $\Phi$ be such that $\rho(\mathcal{L}(\Phi)) = \rho(\mathcal{L}(G))$. Then, beacuse $G$ is bipartite, we have $\rho(\mathcal{L}(G))=2$, implying $\rho(\mathcal{L}(\Phi))=2$, which by Remark  \ref{observation:phi-balanced-2-eigenvalue} implies $\Phi$ is balanced and hence, by using Theorem \ref{theorem: spec and (G,1)}, we have $\spec(\mathcal{L}(G)) = \spec(\mathcal{L}(\Phi))$.
\end{proof}

The following corollaries characterize the bipartiteness in terms of the balancedness, spectrum, and spectral radius of the gain normalized Laplacian matrix.

\begin{corollary}\label{theorem: bipartite and balance}
	Let  $\Phi$ be a connected gain graph. Then the following holds:
	\begin{enumerate}
		\item[(i)] If $\Phi$ is a balanced bipartite graph,  then $2$ is an eigenvalue of $\mathcal{L}(\Phi)$.
		\item[(ii)] If $\Phi$ is balanced and $2$ is an eigenvalue of $\mathcal{L}(\Phi)$, then $\Phi$ is bipartite.
		\item[(iii)] If $\Phi$ is bipartite and $2$ is an eigenvalue of $\mathcal{L}(\Phi)$, then $\Phi$ is balanced.
	\end{enumerate}
\end{corollary}

\begin{proof}
	Proof follows from Theorem \ref{theorem:bipa-balan} and Theorem \ref{observation:eigenvalue of -Phi}.
\end{proof}

\begin{remark}{\rm
		The assumption of balancedness in second part of Corollary \ref{theorem: bipartite and balance} cannot be dropped, as shown in the following example.
		Consider the following Laplacian of a complete gain graph $\Phi$ on $3$ vertices: $$\mathcal{L}(\Phi) = \begin{bmatrix} 1 & 1/2 & 1/2 \\ 1/2 & 1 & 1/2 \\ 1/2& 1/2 & 1 \end{bmatrix}.$$Then, it is easy to check that, $\spec(\mathcal{L}(\Phi)) = \{0.5,0.5,2\}$. Thus,  even though $2$ is an eigenvalue, $\Phi$ is neither bipartite nor balanced.}
\end{remark}

\begin{corollary}
	Let $\Phi = (G, \varphi)$ be a connected $\mathbb{T}$-gain graph. Then the following statemnts hold:
	
	\begin{enumerate}
		\item[(i)] If $\spec(\mathcal{L}(G)) = \spec(\mathcal{L}(\Phi))$ implies $\rho(\mathcal{L}(-G)) = \rho(\mathcal{L}(\Phi))$ for some gain $\varphi$, then $G$ is bipartite.
		\item[(ii)] If $G$ is bipartite, then $\spec(\mathcal{L}(G)) = \spec(\mathcal{L}(\Phi))$ if and only if $\rho(\mathcal{L}(-G))= \rho(\mathcal{L}(\Phi))$ for every gain $\varphi$.
	\end{enumerate}
\end{corollary}
\begin{proof}
	The proof follows from  Theorem \ref{theorem: spec and (G,1)}, Theorem \ref{theorem:bipa-balan} and Remark \ref{observation:phi-balanced-2-eigenvalue}.
\end{proof}

\begin{corollary}
	Let $\Phi = (G, \varphi)$ be a connected $\mathbb{T}$-gain graph. Then the following statements hold:
	\begin{enumerate}
		\item[(i)] If $\spec(\mathcal{L}(G)) = \spec(\mathcal{L}(\Phi))$ implies $\rho(\mathcal{L}(\Phi)) = \rho(\mathcal{L}(-\Phi))$ for some gain $\varphi$, then $G$ is bipartite.
		\item[(ii)] If $G$ is bipartite, then $\spec(\mathcal{L}(G)) = \spec(\mathcal{L}(\Phi))$ implies $\rho(\mathcal{L}(\Phi)) = \rho(\mathcal{L}(-\Phi))$ for every gain $\varphi$.
	\end{enumerate}
\end{corollary}

\begin{proof}
	\begin{enumerate}
		\item[(i)] Let $\spec(\mathcal{L}(G)) = \spec(\mathcal{L}(\Phi))$ implies $\rho(\mathcal{L}(\Phi)) = \rho(\mathcal{L}(-\Phi))$. Then, by Theorem \ref{theorem: spec and (G,1)} that, $\Phi$ is balanced.  Since $\rho(\mathcal{L}(\Phi)) = \rho(\mathcal{L}(-\Phi))$, then $\lambda_n = \alpha_n =  2-\alpha_{{1}}$. By Remark \ref{observation:phi-balanced-2-eigenvalue}, we have $\alpha_n=2$, which implies $\alpha_1=0$. So, by using Corollary \ref{corollary: singularity and balancedness}, $-\Phi$ is bipartite. Thus, it follows from Theorem \ref{theorem:bipa-balan} that $G$ is bipartite.
		
		\item[(ii)]  Let $G$ be a bipartite graph and $\spec(\mathcal{L}(G)) = \spec(\mathcal{L}(\Phi))$. Then it follows from Theorem \ref{theorem: spec and (G,1)} that $\Phi$ is balanced and therefore, by Theorem \ref{theorem:bipa-balan}, $-\Phi$ is also balanced. Thus, using Remark  \ref{observation:phi-balanced-2-eigenvalue}, $\rho(\mathcal{L}(\Phi)) = \rho(\mathcal{L}(-\Phi))=2$.
	\end{enumerate}
\end{proof}

Here we try to provide a generalization of the classic result regarding bipartite graphs that the spectrum of normalized Laplacian of graph is symmetric about one if and only if the graph is bipartite.

\begin{theorem}\label{theorem:symm about 1}
	Let  $\Phi$ be a connected gain graph. If $\Phi$ is bipartite, then all the eigenvalues of $\mathcal{L}(\Phi)$ are symmetric about $1$ (including multiplicities), i.e., for each eigenvalue $\lambda_i, 2-\lambda_i$ is also an eigenvalue of $\mathcal{L}(\Phi)$.
\end{theorem}

\begin{proof}
	
	If $\Phi$ is bipartite, then $\mathcal{A}(\Phi)$ can be expressed as
	$\begin{bmatrix}
		0 & \mathbf{B} \\ \mathbf{B}^* & 0
	\end{bmatrix}$.
	It is clear that the following holds
	$$ \begin{bmatrix}
		-\mathbf{I} & 0 \\ 0 & \mathbf{I}
	\end{bmatrix}\: \begin{bmatrix}
		0 & -\mathbf{B} \\ -\mathbf{B}^* & 0
	\end{bmatrix}\: \begin{bmatrix}
		-\mathbf{I} & 0 \\ 0 & \mathbf{I}
	\end{bmatrix}=\begin{bmatrix}
		0 & \mathbf{B} \\ \mathbf{B}^* & 0
	\end{bmatrix},$$implying that $\mathcal{A}(\Phi)$ and $\mathcal{A}(-\Phi)$ are unitarily similar. Therefore, $\spec(\mathcal{A}(\Phi))=\spec(\mathcal{A}(-\Phi)),$ and hence the result follows.
\end{proof}
\begin{remark}\label{Remark:purely-img-gain}
	The converse of Theorem \ref{theorem:symm about 1} may not be true, as shown in the following counter example.
\end{remark}
Consider the  normalized Laplacian matrix of a complete gain graph $\Phi$ on $3$ vertices given by: $$\mathcal{L}(\Phi) = \begin{bmatrix} 1 & -\text{i}/2 & -\text{i}/2 \\ \text{i}/2 & 1 & -\text{i}/2 \\ \text{i}/2& \text{i}/2 & 1 \end{bmatrix}.$$Then, it is easy to check that, $\spec(\mathcal{L}(\Phi)) = \{1-\frac{\sqrt{3}}{2},1,1+\frac{\sqrt{3}}{2}\}$. Hence, although the eigenvalues of $\mathcal{L}(\Phi)$ are symmetric about $1$, $\Phi$ is not bipartite.

\section{\textsf{Eigenvalue Interlacing}}\label{sec:eig-inter}

Here, we will prove the edge version of the eigenvalue interlacing result and will be closely following the techniques given in \cite{chen2004interlacing} for our proofs.
\begin{lemma}[\cite{chen2004interlacing}]\label{lemma:a,b}
	Suppose that for real numbers $a , b$ and $\gamma$, $$a^2-2\gamma^2 \geq 0, \quad b^2-\gamma^2 >0, \quad \text{and} \quad \frac{a^2}{b^2} \leq 2.$$ Then $$\frac{a^2-2\gamma^2}{b^2-\gamma^2} \leq \frac{a^2}{b^2}.$$
\end{lemma}

\begin{theorem}[Eigenvalue Interlacing]\label{interlacing thm}
	Let  $\Phi$ be a gain graph on $n$ vertices without isolated vertices and $G-e$ be the gain graph obtained from $\Phi$ by removing the edge $e$. Assume that $0 \leq \lambda_1 \leq \lambda_2 \leq \ldots \leq \lambda_n$ and $0 \leq \theta_1 \leq \theta_2 \leq \ldots \leq \theta_n$ are the eigenvalues of $\mathcal{L}(\Phi)$ and $\mathcal{L}(\Phi-e)$, respectively. Then  $$\lambda_{i-1} \leq \theta_i \leq \lambda_{i+1}$$ for each $i = 1, 2, \ldots, n,$ with the convention that $\lambda_0 = 0$ and $\lambda_{n +1} = 2$.
\end{theorem}

\begin{proof}
	Without loss of generality, we suppose that the edge $\overrightarrow{e} = (v_1,v_2)$ is removed, i.e.  $a_{12} = \varphi(e)$ is replaced by $a_{12} = 0$. After deleting the edge $e$, the degrees of $v_1$ and $v_2$ are decreased by one. As shown earlier, we have
	
	\begin{equation*}
		\lambda_k = \max_{\mathbf{y}^{(1)},\mathbf{y}^{(2)}, \ldots,
			\mathbf{y}^{(k-1)} \in \mathbb{C}^n} \min_{\substack{\mathbf{y} \perp \mathbf{y}^{(1)},\mathbf{y}^{(2)}, \ldots, \mathbf{y}^{(k-1)}; \\ \mathbf{y} \neq 0; \: \mathbf{y} \in \mathbb{C}^n}} \frac{\sum_{v_i \sim v_j} |y_i-a_{ij} y_j|^2}{\sum_{i} d_i|y_i|^2}
	\end{equation*}
	
	\begin{equation*}
		= \min_{\mathbf{y}^{(k+1)},\mathbf{y}^{(k+2)}, \ldots,
			\mathbf{y}^{(n)} \in \mathbb{C}^n} \max_{\substack{\mathbf{y} \perp \mathbf{y}^{(k+1)},\mathbf{y}^{(k+2)}, \ldots, \mathbf{y}^{(n)}; \\ \mathbf{y} \neq 0; \: \mathbf{y} \in \mathbb{C}^n}} \frac{\sum_{v_i \sim v_j} |y_i-a_{ij} y_j|^2}{\sum_{i} d_i|y_i|^2}
	\end{equation*}
	
	
	The summmations $\sum_{v_i \sim v_j} |y_i-a_{ij} y_j|^2$ and $\sum_{i} d_i|y_i|^2$ no longer includes the pair $\{1, 2\}$. Thus,
	\begin{equation*}
		\theta_k = \max_{\mathbf{y}^{(1)},\mathbf{y}^{(2)}, \ldots,
			\mathbf{y}^{(k-1)} \in \mathbb{C}^n} \min_{\substack{\mathbf{y} \perp \mathbf{y}^{(1)},\mathbf{y}^{(2)}, \ldots, \mathbf{y}^{(k-1)}; \\ \mathbf{y} \neq 0; \: \mathbf{y} \in \mathbb{C}^n}} \frac{\sum_{v_i \sim v_j} |y_i-a_{ij} y_j|^2-|y_1- \varphi(e)y_2|^2}{\sum_{i} d_i|y_i|^2 -|y_1|^2-|y_2|^2}
	\end{equation*}
	In order to bring the subtracted terms in single variable, so that Lemma \ref{lemma:a,b} can be applied to get the desired result, we relax the minimization constraint by introducing a special relation $y_1 = -\varphi(e) y_2$. So we have,
	\begin{equation*}
		\theta_k \leq \max_{\mathbf{y}^{(1)},\mathbf{y}^{(2)}, \ldots,
			\mathbf{y}^{(k-1)} \in \mathbb{C}^n} \min_{\substack{\mathbf{y} \perp \mathbf{y}^{(1)},\mathbf{y}^{(2)}, \ldots, \mathbf{y}^{(k-1)}; \\ {y_1 = -\varphi(e) y_2}; \\ \mathbf{y} \neq 0; \: \mathbf{y} \in \mathbb{C}^n}} \frac{\sum_{v_i \sim v_j} |y_i-a_{ij} y_j|^2-|y_1- \varphi(e)y_2|^2}{\sum_{i} d_i|y_i|^2 -|y_1|^2-|y_2|^2}
	\end{equation*}
	
	\begin{equation*}
		= \max_{\mathbf{y}^{(1)},\mathbf{y}^{(2)}, \ldots,
			\mathbf{y}^{(k-1)} \in \mathbb{C}^n} \min_{\substack{\mathbf{y} \perp \mathbf{y}^{(1)},\mathbf{y}^{(2)}, \ldots, \mathbf{y}^{(k-1)}, \: {e_1+\varphi(e) e_2}; \\ \mathbf{y} \neq 0; \: \mathbf{y} \in \mathbb{C}^n}} \frac{\sum_{v_i \sim v_j} |y_i-a_{ij} y_j|^2-4|y_1|^2}{\sum_{i} d_i|y_i|^2 -2|y_1|^2}
	\end{equation*}
	
	\begin{equation*}
		\leq \max_{\mathbf{y}^{(1)},\mathbf{y}^{(2)}, \ldots,
			\mathbf{y}^{(k-1)} \in \mathbb{C}^n} \min_{\substack{\mathbf{y} \perp \mathbf{y}^{(1)},\mathbf{y}^{(2)}, \ldots, \mathbf{y}^{(k-1)}, \: {e_1+\varphi(e) e_2}; \\ \mathbf{y} \neq 0; \: \mathbf{y} \in \mathbb{C}^n}} \frac{\sum_{v_i \sim v_j} |y_i-a_{ij} y_j|^2}{\sum_{i} d_i|y_i|^2} \quad \text{(By Lemma \ref{lemma:a,b})}
	\end{equation*}
	
	\begin{equation*}
		\leq \max_{\mathbf{y}^{(1)},\mathbf{y}^{(2)}, \ldots,
			\mathbf{y}^{(k)} \in \mathbb{C}^n} \min_{\substack{\mathbf{y} \perp \mathbf{y}^{(1)},\mathbf{y}^{(2)}, \ldots, \mathbf{y}^{(k)}; \\ \mathbf{y} \neq 0; \: \mathbf{y} \in \mathbb{C}^n}} \frac{\sum_{v_i \sim v_j} |y_i-a_{ij} y_j|^2}{\sum_{i} d_i|y_i|^2} = \lambda_{k+1},
	\end{equation*}where the vectors $e_1, e_2$ are the standard basis vectors. In the second inequality we have used Lemma \ref{lemma:a,b} with $\gamma^2=2|y_1|^2, a^2 = \sum_{v_i \sim v_j} |y_i-a_{ij} y_j|^2 $ and $b^2 = \sum_{i} d_i|y_i|^2$. \\
	
	Similarly, using the min-max version of the Courant–Fischer Theorem \ref{theorem:CF}, we have
	\begin{equation*}
		\theta_k = \min_{\mathbf{y}^{(k+1)},\mathbf{y}^{(k+2)}, \ldots,
			\mathbf{y}^{(n)} \in \mathbb{C}^n} \max_{\substack{\mathbf{y} \perp \mathbf{y}^{(k+1)},\mathbf{y}^{(k+2)}, \ldots, \mathbf{y}^{(n)}; \\ \mathbf{y} \neq 0; \: \mathbf{y} \in \mathbb{C}^n}} \frac{\sum_{v_i \sim v_j} |y_i-a_{ij} y_j|^2-|y_1- \varphi(e)y_2|^2}{\sum_{i} d_i|y_i|^2 -|y_1|^2-|y_2|^2}
	\end{equation*}
	After relaxing the maximization constraint set by introducing a special relation $y_1 = \varphi(e) y_2$, we get
	\begin{equation*}
		\theta_k \geq \min_{\mathbf{y}^{(k+1)},\mathbf{y}^{(k+2)}, \ldots,
			\mathbf{y}^{(n)} \in \mathbb{C}^n} \max_{\substack{\mathbf{y} \perp \mathbf{y}^{(k+1)},\mathbf{y}^{(k+2)}, \ldots, \mathbf{y}^{(n)}; \\ {y_1 = \varphi(e) y_2}; \\ \mathbf{y} \neq 0; \: \mathbf{y} \in \mathbb{C}^n}} \frac{\sum_{v_i \sim v_j} |y_i-a_{ij} y_j|^2-|y_1- \varphi(e)y_2|^2}{\sum_{i} d_i|y_i|^2 -|y_1|^2-|y_2|^2}
	\end{equation*}
	
	
	\begin{equation*}
		\geq \min_{\mathbf{y}^{(k+1)},\mathbf{y}^{(k+2)}, \ldots,
			\mathbf{y}^{(n)} \in \mathbb{C}^n} \max_{\substack{\mathbf{y} \perp \mathbf{y}^{(k+1)},\mathbf{y}^{(k+2)}, \ldots, \mathbf{y}^{(n)}, \: {e_1-\varphi(e) e_2}; \\ \mathbf{y} \neq 0; \: \mathbf{y} \in \mathbb{C}^n}} \frac{\sum_{v_i \sim v_j} |y_i-a_{ij} y_j|^2}{\sum_{i} d_i|y_i|^2}
	\end{equation*}
	
	\begin{equation*}
		\geq \min_{\mathbf{y}^{(k)},\mathbf{y}^{(k+2)}, \ldots,
			\mathbf{y}^{(n)} \in \mathbb{C}^n} \max_{\substack{\mathbf{y} \perp \mathbf{y}^{(k)},\mathbf{y}^{(k+2)}, \ldots, \mathbf{y}^{(n)}; \\ \mathbf{y} \neq 0; \: \mathbf{y} \in \mathbb{C}^n}} \frac{\sum_{v_i \sim v_j} |y_i-a_{ij} y_j|^2}{\sum_{i} d_i|y_i|^2} = \lambda_{k-1},
	\end{equation*}
	Hence $$\lambda_{k-1} \leq \theta_k \leq \lambda_{k+1}$$ with the convention $\lambda_0 = 0$ and $\lambda_{n+1} = 2$. The values of $\lambda_0$ and $\lambda_{n+1}$ have been chosen to make the upper and lower bounds true for $\theta_{0}$ and $\theta_{n}$. The cases when $e = (v_2,v_1)$ and $e$ is a digon can be proved similarly.
\end{proof}

\begin{remark}
	Theorem \ref{interlacing thm} also holds when $\Phi$ has isolated vertices. In this case some additional zero eigenvalues exist. The removal of an edge is only taken on the subgraph without isolated vertices. The eigenvalue interlacing relation can be considered on the corresponding submatrix.
\end{remark}

The following is a direct consequence of Theorem \ref{interlacing thm}.

\begin{corollary}
	Let $\Phi$ be a gain graph and $\mathcal{H}$ be a spanning subgraph of $\Phi$ such that $|E(\Phi-\mathcal{H})| \leq t$ for some integer $t$. Assume that $\lambda_1 \leq \lambda_2 \leq \ldots \leq \lambda_n$ and $\theta_1 \leq \theta_2 \leq \ldots \leq \theta_n$ are eigenvalues of $\mathcal{L}(\Phi)$ and $\mathcal{L}(\mathcal{H})$ , respectively. Then  $$\lambda_{k-t} \leq \theta_k \leq \lambda_{k+t} \quad \text{for each} \quad k = 1, 2, \ldots, n,$$ with the convention\begin{align*}
		\lambda_{1-t} &= \lambda_{2-t} = \ldots = \lambda_0 = 0,\\
		\lambda_{n+1} &= \lambda_{n+2} = \ldots = \lambda_{n+t} = 2.\end{align*}
\end{corollary}

\section{\textsf{Characteristic Polynomial of $\mathcal{L}(\Phi)$}}\label{sec:char-poly}
In this section, we will first recall some known definitions. Then, we will compute the coefficients of the characteristic polynomials in terms of the gains of the edges. In \cite[section 4]{hermitian}, the authors considered characteristic polynomials for normalized Hermitian Laplacian matrix, while in \cite{ran-raj-ani-LAMA}, characteristic polynomials for gain adjacency matrix was considered. We will be closely following their proof techniques to prove the results in this section for the normalized gain Laplacian matrix.

Let the characteristic polynomial of $\mathcal{L}(\Phi)$ be $$\Gamma(x, \mathcal{L}(\Phi)) = \det[x \mathbf{I}-\mathcal{L}(\Phi)]= x^n+b_1 x^{n-1}+b_2 x^{n-2}+ \ldots + b_n.$$ The characteristic polynomial of $\mathcal{L}(\Phi)$ can be also expressed as $$\Gamma(x, \mathcal{L}(\Phi)) = \det[(x-1) \mathbf{I}+\mathbf{D}^{-\frac{1}{2}}\: \mathbf{A}(\Phi) \: \mathbf{D}^{-\frac{1}{2}}]= (x-1)^n+c_1 (x-1)^{n-1}+c_2 (x-1)^{n-2}+ \ldots + c_n.$$
Here we will investigate the coefficients of the characteristic polynomial of $\mathcal{L}(\Phi)$. Before moving forward, let us review some of the basic definitions and notations used throughout the section. A \textit{dissection
	graph} is a directed graph such that every component is a vertex or an edge or a cycle. A special dissection graph, known as \textit{elementary graph}, is a directed graph such that every component is an edge or a cycle. A \textit{(real) spanning elementary subgraph} of a gain graph $\Phi$ is an elementary subgraph such that it contains all vertices of $G$ and all its cycles are real. The set of all elementary spanning subgraphs $H$ of $\Phi$ is denoted by $\mathcal{H}(\Phi)$, and $\mathcal{C}(H)$ denotes the collection of all elementary cycles in $H$. The rank $r(\Phi)$ and the corank $s(\Phi)$ of a gain graph $\Phi$ are $r(\Phi) = n - p(\Phi)$, $s(\Phi) = m - n + p(\Phi)$, where $n, m$ and $p(\Phi)$ are the number of vertices, the number of edges and the number of components of $\Phi$, respectively.

Next, we provide an expression of all the coefficients of the characteristic polynomial of $\mathcal{L}(\Phi)$ in terms of elementary subgraphs of $\Phi$. We start by recalling the well known Harary's determinant formula for the adjacency matrices of graphs.
\begin{theorem}\label{harary-form}\cite{harary-det-form}
	Let $\mathbf{A}$ be the adjacency matrix of a graph $G$. Then $\det(\mathbf{A}(G)) = \sum (-1)^{r(H)} 2^{s(H)},$ where the summation is over all elementary spanning subgraphs $H$ of $G$.
\end{theorem}

For $z \in \mathbb{C}$, let $\Re(z)$ denote the real part of $z$.  The proof of the following two results are known for a much more wider class. We need the following particular cases here.
\begin{theorem} \cite[Theorem 2.2]{germina2012balance}
	\label{gain-har-det}
	Let $\Phi$ be a gain graph with the underlying graph $G$. Then
	\begin{equation}
		\det (\mathbf{A}(\Phi))=\sum_{H\in\mathcal{H}(G)}(-1)^{n-p(H)}2^{c(H)}\prod_{C\in \mathcal{C}(H)}\Re(C),
	\end{equation}
	where $p(H)$ is the number of components in $H$ and  $c(H)$ is the number of cycles in $H$.
\end{theorem}
\begin{theorem}\cite[Corollary 2.3]{germina2012balance}
	\label{gain-sachs-det}Let $\Phi$ be any gain graph with the underlying graph $G$. Let $P_\Phi(x)=x^n+a_1x^{n-1}+\cdots+a_n$ be the characteristics polynomial of $\Phi$. Then
	$$a_i=\sum_{H\in\mathcal{H}_i(G)}(-1)^{p(H)}2^{c(H)}\prod_{C\in \mathcal{C}(H)}\Re(C),$$
	where $\mathcal{H}_i(G)$ is the set of elementary subgraphs of $G$ with $i$ vertices.
\end{theorem}

\begin{theorem}\label{theorem:charpoly-det-adj}
	Let $\mathbf{A}(\Phi)$ be the    adjacency matrix of a gain graph $\Phi = (G, \varphi)$. Then $$\det(\mathbf{A}(\Phi)) = \sum_{H} (-1)^{r( H )}2^{s( H )}\: \mathop{\Pi}_{C \in \mathcal{C}(H)} \Re(\varphi(C )),$$where the sum is over all spanning  elementary subgraphs $ H $ of $\Phi$.
\end{theorem}

\begin{proof}
	Proof follows from the proofs of Theorem \ref{harary-form} and Theorem \ref{gain-har-det}.
\end{proof}
It is easy to see that, the above theorem extends  \cite[Theorem 2.7]{liu2015hermitian} for gain graphs.
\begin{theorem}\label{cor: CharPoly adjacency}
	Let $\Phi$ be a gain graph and $\Gamma(x, \mathbf{A}(\Phi)) = x^n+a_1 x^{n-1}+a_2 x^{n-2}+ \ldots + a_n$ be be its characteristic polynomial of $\mathbf{A}(\Phi)$. Then $$(-1)^{k}a_k = \sum_{\mathcal{H}} (-1)^{r(\mathcal{H})}2^{s(\mathcal{H})}\:  \mathop{\Pi}_{C \in \mathcal{C}(H)} \Re(\varphi(C )),$$ where the sum is over all the elementary subgraphs $\mathcal{H}$ of $\Phi$ with $k$ vertices.
\end{theorem}

\begin{proof}
	The proof follows from the proofs of Sachs coefficient theorem and Theorem \ref{gain-sachs-det}.
\end{proof}

It is easy to see that the above theorem extends  \cite[Theorem 2.8]{liu2015hermitian}   for gain graphs.

Next, we discuss the combinatorial description of coefficients of the characteristic polynomials of gain normalized Laplacian matrices. Similar to \cite{hermitian}, we now introduce a new polynomial in order to attain the required results: $$\det[(x-1) \mathbf{I}-\mathbf{D}^{-\frac{1}{2}}\: \mathbf{A}(\Phi) \: \mathbf{D}^{-\frac{1}{2}}]= (x-1)^n+c^{'}_1 (x-1)^{n-1}+c^{'}_2 (x-1)^{n-2}+ \ldots + c^{'}_n.$$ Then $$\det[(x-1) \mathbf{I}+\mathbf{D}^{-\frac{1}{2}}\: \mathbf{A}(\Phi) \: \mathbf{D}^{-\frac{1}{2}}]= (-1)^n \det[(1-x) \mathbf{I}-\mathbf{D}^{-\frac{1}{2}}\: \mathbf{A}(\Phi) \: \mathbf{D}^{-\frac{1}{2}}]$$ and $(-1)^{k}c^{'}_k$ equals to the sum of all $k \times k$ minors of $\mathbf{D}^{-\frac{1}{2}}\: \mathbf{A}(\Phi) \: \mathbf{D}^{-\frac{1}{2}}$. It follows that

\begin{lemma}\label{lemma: CharPoly phi}
	Let $\Phi$ be a gain graph on $n$ vertices and $\Gamma(x, \mathcal{L}(\Phi))= (x-1)^n+c_1 (x-1)^{n-1}+c_2 (x-1)^{n-2}+ \ldots + c_n$. Then $c_k$,  $k=1,2,\ldots, n$, is the sum of all $k \times k$ minors of $\mathbf{D}^{-\frac{1}{2}}\: \mathbf{A}(\Phi) \: \mathbf{D}^{-\frac{1}{2}}$.
\end{lemma}

\begin{theorem}
	Let $\Phi$ be a gain graph on $n$ vertices and $\Gamma(x, \mathcal{L}(\Phi))= (x-1)^n+c_1 (x-1)^{n-1}+c_2 (x-1)^{n-2}+ \ldots + c_n$. Then $$c_k = \sum_{H} (-1)^{r(H) }2^{s(H)} \frac{\mathop{\Pi}_{C \in \mathcal{C}(H)} \Re(\varphi(C ))}{\Pi_{v \in V(H)}d_{\Phi} (v)}$$ where the summation is over all the elementary subgraphs $H$ of $\Phi$ with $k$ vertices and $d_{\Phi} (v)$ is the degree of the vertex $v$ in $\Phi$.
\end{theorem}

\begin{proof}
	By Lemma \ref{lemma: CharPoly phi}, $c_k$ is the sum of all $k \times k$ minors of $\mathbf{D}^{-\frac{1}{2}}\: \mathbf{A}(\Phi) \: \mathbf{D}^{-\frac{1}{2}}$. Each such $k \times k$ minors of $\mathbf{D}^{-\frac{1}{2}}\: \mathbf{A}(\Phi) \: \mathbf{D}^{-\frac{1}{2}}$ is the product of the corresponding $k \times k$ minors of $\mathbf{D}^{-\frac{1}{2}}, \mathbf{A}(\Phi), \mathbf{D}^{-\frac{1}{2}}$, respectively. Moreover, any $k \times k$ minor of $\mathbf{A}(\Phi)$ is the determinant of gain adjacency matrix of an induced subgraph of $\Phi$ with $k$ vertices. So the result holds by Theorem \ref{cor: CharPoly adjacency}.
\end{proof}

The following result is a counterpart of the well-known Sachs coefficient theorem for the gain normalized Laplacian. The proof is similar to that of  \cite[Theorem 4.2]{hermitian}. For the sake of completeness, we include a proof here.

\begin{theorem}
	Let $\Phi$ be a gain graph on $n$ vertices and $\Gamma(x, \mathcal{L}(\Phi)) = x^n+b_1 x^{n-1}+b_2 x^{n-2}+ \ldots + b_n$.  Then for $k = 1, 2, \ldots, n$ $$(-1)^{k}b_k = \sum_{H} (-1)^{r(H)+o(H)} \frac{2^{s(H)}}{D_{H}}\:
	{\mathop{\Pi}_{C \in \mathcal{C}(H)} \Re(\varphi(C ))}$$
	where the sum is over all dissection subgraphs $H$ of $\Phi$ with $k$ vertices, $o(H)$ denotes the number of odd cycles in $H$, $D_{H}=\Pi_{v \in V(H), d_{H}(v) \neq 0}d_{\Phi} (v)$, $d_{\Phi} (v)$ is the degree of the vertex $v$ in $\Phi$.
\end{theorem}

\begin{proof}
	As stated earlier, $(-1)^{k}c_k = \sum_k M_k$ where $M_k$ is the $k \times k$ minor of $\mathcal{L}(\Phi)$. Let $B = \{i_1, i_2, \ldots, i_k\}$ be a subset of $ V(G)$ with $k$ vertices, and $H$ be the subgraph induced on $B$. As $M_k = \sum_\sigma \sgn(\sigma) \mathcal{L}_{i_1 \sigma(i_1)} \mathcal{L}_{i_2 \sigma(i_2)} \ldots \mathcal{L}_{i_k \sigma(i_k)}$, so we need to consider only the terms with  $ \mathcal{L}_{i_1 \sigma(i_1)} \mathcal{L}_{i_2 \sigma(i_2)} \ldots \mathcal{L}_{i_k \sigma(i_k)}$ is non-zero. It is clear that,  the term $\mathcal{L}_{i_j \sigma(i_j)} \neq 0$ if and only if $\sigma(i_j) = i_j$, or $\sigma(i_j) \neq i_j$ and the vertices $i_j$ and  $\sigma(i_j)$ are adjacent in $H$. Let $\sigma$ be a permutation corresponding to the non-zero term. Then $\sigma$ can be written as the product of disjoint cycles, say $$(i_1,i_2, \ldots, i_s) (i_{s+1},i_{s+2}, \ldots, i_{t}) \ldots (i_m) (i_{m+1}) \ldots (i_k).$$ Let $f(H)$ be the number of fixed vertices under $\sigma, o(H)$ is the number of odd cycles in $\sigma$ and $c_l$ be the number of cycles of length $l$. Then $f(H)+\sum_l l c_l = n$ and $n - f(H) \equiv o(H)(\mod2)$. As $\sgn(\sigma) = (-1)^{e(H)}$ where $e(H)$ is the number of even cycles in $\sigma$, we have $$r(H) = n - f(H) - o(H) - e(H) \equiv e(H)(\mod 2)$$ and hence $$\sgn(\sigma) = (-1)^{r(H)}.$$ A $2$-cycle in $\sigma$ corresponds to an edge of $H$. For  each cycle $(i_1 i_2 \ldots i_r) (r \geq 3)$ in $\sigma$, there exists a cycle $i_1 i_2 \ldots i_r i_1$ in $H$ corresponding to it.  So each non-zero term gives rise to a dissection graph. That is, $H$ is a dissection subgraph of $\Phi$. Let $B_1$ be the set of fixed vertices under $\sigma$. Let $B_2$ be the set of edges in $H$ corresponding to the $2$-cycles in the disjoint cycle factorization of $\sigma$. Let $C_1, C_2, \ldots, C_l$ be all cycles of $H$ corresponding to the cycles of length more than $2$ in the disjoint cycle factorization of $\sigma$. Note that any $2$-cycles $(i_s, i_t)$ in $\sigma$ corresponds to the nonzero factor $\mathcal{L}_{i_s i_t}\mathcal{L}_{i_t i_s} = \frac{1}{d_{i_s} d_{i_t}}$. Any $r$-cycle $(i_{1} i_{2} \ldots i_{r}) (r \geq 2)$ in $\sigma$ corresponds to the non-zero factor of $\mathcal{L}_{i_1 i_2} \mathcal{L}_{i_2 i_3}  \ldots \mathcal{L}_{i_r i_1} =\frac{(-1)^r \varphi(C)}{d_{i_1} d_{i_2} \ldots d_{i_r}},$ where $C$ is the cycle $i_1 i_2 \ldots i_r i_1
	$ in $H$ and $\varphi(C)$ is the gain of the cycle $C$. Thus
	$$\sgn(\sigma) \mathcal{L}_{i_1 \sigma(i_1)} \mathcal{L}_{i_2 \sigma(i_2)} \ldots \mathcal{L}_{i_k \sigma(i_k)}$$
	\begin{align*}
		&= (-1)^{r(H)} \mathop{\Pi}_{i_j \in B_1}\mathcal{L}_{i_j i_j} \mathop{\Pi}_{(i_j, \sigma(i_j)) \in B_2}(\mathcal{L}_{i_j \sigma(i_j)} \mathcal{L}_{\sigma(i_j) i_j}) \left(\mathop{\Pi}_{(i_j, \sigma(i_j)) \in E(C_l)} \mathcal{L}_{i \sigma(i)} \right) \cdots \left(\mathop{\Pi}_{(i_j, \sigma(i_j)) \in E(C_l)} \mathcal{L}_{i \sigma(i)} \right)\\ &= (-1)^{r(H)} \left(\mathop{\Pi}_{(i_j, \sigma(i_j)) \in B_2} \frac{1}{d_{i_j} d_{\sigma(i_j)}}\right) \cdot \left(\mathop{\Pi}_{i=1}^l \frac{(-1)^{g(C_i)} \varphi(C_i)}{D(C_i)}\right)\\ &= (-1)^{r(H)+o(H)}\frac{1}{D_{H}} \mathop{\Pi}_{i=1}^l \varphi(C_i),
	\end{align*} where $g(C_i)$ is the length of $C_i$, $D(C_i)=\Pi_{v \in V(C_i)}d_{\Phi} (v), D_{H}=\Pi_{v \in V(H), d_{H}(v) \neq 0}d_{\Phi}(v)$ and $d_{\Phi}(v)$ (resp. $d_{H}(v)$) is the degree of the vertex $v$ in $\Phi$ (resp. $H$).\\
	
	Each dissection graph $H$ with $k$ vertices gives rise to several permutations $\sigma$ for which the corresponding term in the
	minor expansion is non-zero. For a cycle in $H$, there are two ways to choose the corresponding cycle in $\sigma$. For example, the cycle $i_{j_1} i_{j_2} \ldots i_{j_r} i_{j_1}$  corresponds to $(i_{j_1} i_{j_2} \ldots i_{j_r})$  and $(i_{j_r} i_{j_{r-1}} \ldots i_{j_1})$. So the number of such $\sigma$ arising from a given $H$ is $2^{c(H)}$ where $c(H)$ is the number of cycles in $H$. Note that $s(H) = c(H)$ for dissection graph $H$. Moreover, if for one direction of a permutation, a cycle $C$ in $H$ has the value $\varphi(C)$, then for the other direction the cycle has the value $\overline{\varphi}(C)$. Thus each dissection graph $H$ contributes $(-1)^{r(H)+o(H)} \frac{1}{D_H} \mathop{\Pi}_{C \in \mathcal{C}(H)} [\varphi(C)) + \overline{\varphi(C)}]$. This completes the proof.
\end{proof}

\textbf{Acknowledgements.} The authors would like to thank the anonymous reviewer for the careful reading of the manuscript.  M. Rajesh Kannan would like to thank the Department of Science and Technology, India, for financial support through the projects MATRICS (MTR/2018/000986) and Early Career Research Award (ECR/2017/000643).

\bigskip
{\bf \large Contact Information}\\

\begin{tabular}{lcl}
M. Rajesh Kannan	&& Department of Mathematics, Indian Institute of Technology Kharagpur\\
rajeshkannan1.m@gmail.com  && Kharagpur 721 302, India\\
				&& \\
Navish Kumar	&&  Department of Humanities and Social Sciences, Indian Institute of Technology Kharagpur\\
navish.iitkgp@gmail.com  && Kharagpur 721 302, India\\
				&& \\				
Shivaramakrishna Pragada	&&  Department of Aerospace Engineering, Indian Institute of Technology Kharagpur\\
shivaramkratos@gmail.com  && Kharagpur 721 302, India\\
&& \\			
\end{tabular}

\end{document}